\documentclass[12pt,reqno]{article}
\usepackage{comment}
\usepackage{color}
\usepackage{fullpage}
\usepackage{float}
\usepackage{lineno} 
\usepackage{graphics,amsmath,amssymb}
\usepackage{amsthm}
\usepackage{amsfonts}
\usepackage{latexsym}
\usepackage{epsf}
\usepackage{mathtools}

\usepackage{amssymb}
\usepackage{graphicx}

\newtheorem{theorem}{Theorem}

\newtheorem{proposition}{Proposition}

\newtheorem{example}{Example}

\newtheorem{observation}{Observation}

\usepackage{hyperref}

\def\nn{{\mathbb N}}

\def\BB{{\mathcal B}} \def\CC{{\mathcal C}}

 \def\SS{{\mathcal S}}

\def\supp{{\rm supp}}
\def\fix{{\rm fix}}

\def\t{\tau}
\def\be{\beta}
\def\al{\alpha}

\title{Integer sequences and $k$-commuting permutations}
\author{Luis Manuel Rivera\footnote{Unidad Acad\'emica de Matem\'aticas, Universidad Aut\'onoma de Zacatecas, Calzada Solidaridad entronque Paseo a la Bufa, Zac., Zacatecas, M\'exico, 98000.  E-mail: luismanuel.rivera@gmail.com.}
}
\date{}              
\begin{document}
\maketitle
\begin{abstract}
 Let $\be$ be any permutation on $n$ symbols, and let $c(k, \be)$ be the number of permutations that $k$-commute with $\be$. The cycle type of a permutation $\be$ is a vector $(c_1, \dots, c_n)$ such that $\be$ has exactly $c_i$ cycles of length $i$ in its disjoint cycle factorization. In this article we obtain formulas for $c(k, \be)$, for some cycle types. We also express these formulas in terms of integer sequences as given in ``The On-line Encyclopedia of Integer Sequences" (OEIS). For some of these sequences, we obtain either new interpretations or relationships with sequences in the OEIS database.
 \end{abstract}


\section{Introduction}
Let $S_n$ denote the group of permutations on the set $\{1, \dots, n\}$. For a  nonnegative integer, $k$, we say that two permutations, $\al, \be \in S_n$ {\it $k$-commute} (resp. {\it $(\leq k)$-commute}) if $H(\al\be, \be\al)=k$ (resp. $H(\al\be, \be\al) \leq k$), where $H$ denotes the Hamming metric between permutations (see Deza and Huang~\cite{dezahuang} for a survey about metrics on permutations). For a given permutation, $\be$, and a nonnegative integer, $k$, let $c(k, \be)$ (resp. $c(\leq k, \be)$) denote the number of permutations that $k$-commute (resp. $(\leq k)$-commute) with $\be$. It is known~\cite{morriv} that $c(k, \be)$ only depends on the cycle type of $\be$. In \cite{morriv}, we began the study of $k$-commuting permutations and presented the first partial results for the problem of computing $c(k, \be)$. The original motivation for studying these types of questions was to develop tools in order to solve the open problem of determining the stability of the equation $xy=yx$ in permutations\footnote{During the review process,  Arzhantseva and P\u{a}unescu \cite{arz-pau}, proved that the equation $xy=yx$ is stable in permutations using ultraproduct techniques.} (see \cite{glebriv1, morriv} for the definitions). Also, the comprehensive problem of determining $c(k, \be)$ is interesting in the context of integer sequences. The problem of computing explicit formulas for $c(k, \be)$, for any $k$ and any $\be$, seems to be a difficult task in general, however we think that for some choices of $k$ and $\be$ the problem is manageable. For example, in \cite{morriv} it was found a characterization of permutations that $k$-commute with a given permutation $\be$ and, as consequence, formulas for $c(k, \be)$ when $k=3, 4$ were obtained. For $k > 4$, the problem was solved only when $\be$ is either a transposition,  a fixed-point free involution, or an $n$-cycle. Based on this previous work, we continue this line of research and we obtain formulas when $\be$ is either a $3$-cycle, a $4$-cycle, an $(n-1)$-cycle, and other special cases. We also present an upper bound for $c(\leq k, \be)$, when $\be$ is any $n$-cycle, and explicit formulas for $c(\leq k, \be)$, when $\be$ is either a transposition, a $3$-cycle, or a $4$-cycle.

Surprisingly, in~\cite{morriv} and in this work, for some values of $k$ and some permutations $\be$, the number $c(k, \be)$ is shown to be related to some integer sequences in ``The On-line Encyclopedia of Integer Sequences" (OEIS). Consequently, a new interpretation for some of these sequences is pointed out. For example:
\begin{itemize}
\item  Sequence A208528 \cite{deutsch, nacin} corresponds to the ``number of permutations on $n > 1$ having exactly 3 points, $P$, on the boundary of their bounding square." This sequence also counts the number of permutations on $n$ symbols that $3$-commute with a transposition. In Section~\ref{trans-bij} we provide an explicit relationship between these two structures.  
\item For sequence A001044, the formula is $a(n)=(n!)^2$ and has different interpretations \cite{slo-guy}. In this article we show that $a(n)$ also counts the number of permutations on $2n$ symbols that $2n$-commute with a permutation on $2n$ symbols that has exactly $n$ fixed points.
\item  Let $a(n)$ denote the sequence A004320. For $n \geq 3$, $a(n-2)$ is the number of permutations on $n$ symbols that $3$-commute with an $n$-cycle. 
\item Sequence A001105 is given by $a(n)=2n^2$ and has several interpretations. (See~\cite{bern}). We note that $a(n)$ is also the number of permutations on $2n$ symbols that $0$-commute, i.e., that commute, with a permutation whose disjoint cycle factorization consists of a product of two $n$-cycles.
\item Let $a(n)$ denote sequence A027764. For $n \geq 3$, $a(n)$ is also the number of permutations on $n+1$ symbols that $4$-commute with an $(n+1)$-cycle. 
\item Sequence A000165 is the double factorial of even numbers, that is, $a(n)=(2n)!!=2^nn!$, but there are more interpretations. (See~\cite{slo}). We note that $a(n)$ also counts the number of permutations on $2^n$ symbols that $0$-commute with a fixed-point free involution (a permutation whose disjoint cycle factorization consists of $n$ transpositions).   
\end{itemize} 

For other cases, we obtain the following expressions:
\begin{itemize}
\item Let $a(n)$ denote the sequence A027765. We show that ``$8a(n)$ is the number of permutations on $(n+1)$ symbols that $5$-commute with an $(n+1)$-cycle."
\item Let $a(n)$ and $b(n)$ denote sequences A016777 and A052560, respectively. For $n \geq 3$, we show that $a(n-3) \times b(n-3)$ is the number of permutations on $n$ symbols that $3$-commute with a $3$-cycle.  
\item Let $a(n)$ and $b(n)$ denote the sequences A134582 and A052578, respectively. For $n \geq 5$ we show that $a(n-3) \times b(n-4)$ is the number of permutations on $n$ symbols that $5$-commute with a $4$-cycle.
\end{itemize}

The idea of computing our formulas in terms of sequences in the OEIS database was motivated by the Future Projects section in the OEIS Wiki page~\cite{oeis}. One project is searching sequences in books, journals, and preprints, as suggested in the OEIS Wiki page: ``What needs to be done: Scan these journals, books and preprints looking for new sequences or additional references for existing sequences." After finding some equations for $c(k, \be)$, for some $k$ and $\be$, we decided to work through these results to find all the sequences that occur in these formulas. First, we factorized a formula, and then we looked for factors in the OEIS database. Our results were gathered in tables and used to determine new relationships in the OEIS database. 

In Section~\ref{basicdef}, we give some definitions and notation that is used through the paper. In Section~\ref{someformulas}, we present formulas for $c(k, \be)$  when $\be$ is either a $3$-cycle, a $4$-cycle, an $(n-1)$-cycle, or other special cases. We also present an upper bound for $c(\leq k, \be)$ when $\be$ is an $n$-cycle.  In Section~\ref{relations}, we show some relationships on the number $c(k, \be)$ with integer sequences in the OEIS database, where $\be$ is either a transposition, a $3$, $4$, $n$, and $(n-1)$-cycle, or is the product of specific cycles. In Section~\ref{finalc}, we present our final comments and discuss the case $k=0$ for some permutations.
 
\section{Basic definitions}\label{basicdef}

In this section, we give some definitions and notation that used throughout the paper. Let $[n]$ denote the set $\{1, \dots, n\}$ whose elements are called {\it points}. A permutation of $[n]$ is a bijection from $[n]$ onto $[n]$. We use $S_n$ to denote the group of all permutations of $[n]$. We  write $\pi =p_1 \dots  p_n $ for the one-line notation of $\pi \in S_n$, i.e., $\pi(i)=p_i$ for every $i \in [n]$, and $\tau=(a_1  \dots a_m)$ for an $m$-cycle in $S_n$, i.e., $\tau(a_i)=a_{i+1}$, for $1\leq i \leq m-1$, $\tau(a_m)=a_1$, and $\tau(a)=a$ for every $a \in [n] \setminus \{a_1, \dots, a_{m}\}$. The {\it support}  $\supp(\pi)$ of $\pi \in S_n$ is  $\{x \in [n] \colon~\pi(x) \neq x\}$, and the set of fixed points $\fix(\pi)$ of $\pi$ is $[n] \setminus \supp(\pi)$. The product $\al\be$ of permutations is computed first by applying $\be$ and then $\al$. We say that $\pi$ has cycle $\pi'$ or that $\pi'$ is a {\it cycle of} $\pi$, if $\pi'$ is a factor in the disjoint cycle factorization of $\pi$. The {\it cycle type} of a permutation, $\be$,  is a vector $(c_1, \dots, c_n)$ such that $\be$ has exactly $c_i$ cycles of length $i$ in its disjoint cycle factorization.  
The {\it Hamming metric} between permutations $\al, \be \in S_n$, denoted $H(\al, \be)$, is $|\{ a \in [n] \colon~ \al(a) \neq \be(a)\}|$. We say that $\al$ and $\be$ {\it $k$-commute} if $H(\al\be, \be\al)=k$. We say that  $a\in[n]$ is a {\it good commuting point}  (or {\it bad commuting point}) of $\al$ and $\be$ if $\al\be(a)= \be\al(a)$ (or $\al\be(a)\neq \be\al(a)$). 
Usually, we abbreviate good commuting points (or bad commuting points) with {\it g.c.p.} (or {\it b.c.p.}) and we sometimes omit the reference to $\al$ and $\be$. Let $C(k, \be)=\{\al \in S_n \colon~H(\al\be, \be\al)=k\}$ and $c(k, \be)=| C(k, \be)|$. Let $C(\leq k, \be)=\{\al \in S_n \colon~H(\al\be, \be\al)\leq k \}$ and  $c(\leq k, \be)=| C(\leq k, \be)|$. 
In this paper, we use the convention, $m \bmod m = m$, for every positive integer $m$.

\subsection{Blocks in cycles}
This section is based on Section 2.1 in \cite{morriv}. 
Let $\pi \in S_n$. A {\it block} in a cycle $\pi'=(a_1 \dots a_m)$ of $\pi$ is a consecutive nonempty substring $a_i \dots  a_{i+l}$ of $a_i \dots  a_{i-1}$ where $(a_i \dots  a_{i-1})$ is one of the $m$ equivalent expressions of  $\pi'$ (the sums are taken modulo $m$). This definition was motivated by the notion of a block when the permutation is written in one-line-notation as in \cite{bona3, chris}. The {\it length} $|A|$ of block $A=a_i \dots  a_{i+l}$ is the number of elements in the string $A$, i.e., $|A|=l+1$.  
 Two blocks are {\it disjoint} if they do not have points in common. 
The {\it product} $AB$ of two disjoint blocks, $A$ and $B$, not necessarily from the same cycle of $\pi$, is defined as the usual concatenation of strings. Notice that $AB$ is not necessarily a block in a cycle of $\pi$.  If $(a_1 \dots a_m)$ is a cycle of $\pi$ we write $(A_1 \dots A_k)$ to mean that $A_1 \dots A_k=a_{i} \dots a_{i-1}$, where $(a_1 \dots a_m)=(a_{i}  \dots a_{i-1})$. A {\it block partition} of cycle $\pi'$ is a set $\{A_1, \dots , A_l\}$ of pairwise disjoint blocks in $\pi'$ such that there exists a block product $A_{i_1}\dots  A_{i_l}$ such that $\pi'=(A_{i_1}\dots  A_{i_l})$. 
Let $A= P_{1}\dots  P_{k}$ be a block product of $k$ pairwise disjoint blocks, not necessarily from the same cycle of $\pi$, and let $\tau$ be a permutation in $S_k$. The {\it block permutation} $\phi_\t(A)$ of $A$, induced by $\tau$, is defined as the block product $P_{\tau(1)}\dots  P_{\t(k)}$. 

\begin{example}
Let $\pi=(1\;2\;3\;4\;5)(6\;7\;8\;9) \in S_9$. Some blocks in cycles of $\pi$ are $P_1=1\;2\;3$, $P_2=4$, $P_3=6$, $P_4=7\;8\;9$. One block in $(1\;2\;3\;4\;5)$ is $P_5=3\;4\;5\;1\;2$. The set of blocks $\{P_3, P_4\}$ is a block partition of $(6\;7\;8\;9)$. The product $P_1P_2$ is a block in $(1\;2\;3\;4\;5)$. The product $P_1P_3=1\;2\;3\;6$ is not a block in any cycle of $\pi$.    
\end{example}

\begin{example}
Let $\pi'=(3\;4\;1\;2\;6)$ be a cycle of $\pi  \in S_6$ and $P=P_1P_2P_3$ the block product given by $P_1=3\;4$, $P_2=1$ and $P_3=2\;6$.  Let $\al=(2 \;3 \;1) \in S_3$. The block permutation $\phi_\al(P)$  is $P_2P_3P_1=1\;2\;6\;3\;4$, where, for example, $P_3P_1=2\;6\;3\;4$ is a block in $\pi'$. If $\tau=(1\;3)$, then $\phi_\tau(P)=P_3P_2P_1=2\;6\;1\;3\;4$, where $P_3P_2=2\;6\;1$ is not a block in $\pi'$. 
 \end{example}
If $\pi \in S_n$ and $X \subseteq [n]$, the restriction function of $\pi$ to set $X$ is denoted by $\pi|_X$, i.e., $\pi|_X \colon X \to X$ is defined as $\pi|_X(a)=\pi(a)$ for every $a \in X$. Let $\al, \be \in S_n$. Let $\be'=(b_1\dots  b_m)$ be a cycle of $\be$. As $\al \be' \al^{-1}=(\al(b_1) \dots   \al(b_m))$ (see, e.g., \cite[Prop.\ 10, p.\ 125]{dum}) we use the following notation for $\al|_{\supp(\be')}$ 
\begin{equation}\label{alfares}
\al|_{\supp(\be')}=\left(
\begin{array}{ccccccccc}
 b_1 & \dots   & b_m \\
\al(b_1) &  \dots   & \al(b_m)
 \end{array}
\right).
\end{equation}
If $\al|_{\supp(\be')}$ is written as in (\ref{alfares}), we write 
\begin{equation}\label{alfares2}
\al|_{\supp(\be'),  k}=\left(\begin{array}{cccc}
B_1  \dots   B_k   \\
 J_1     \dots   J_k  
 \end{array} \right),
\end{equation}  
to mean that $B_1  \dots   B_k=b_1\dots b_m$ and $J_1, \dots , J_k$ are blocks in cycles of $\be$, where $J_1 \dots  J_k=\al(b_1) \dots  \al(b_m)$ and  $|J_i|=|B_i|$, for $1\leq i \leq k$. This notation is called  {\it block notation} (with respect to $\be$) of $\al|_{\supp(\be')}$. Notice that this notation depends on the particular selection of one of the $m$ equivalent cyclic expressions of $\be$: $(b_1\dots b_m)$, $(b_2 \dots b_1)$, \dots, $(b_m \dots b_{m-1})$. Sometimes we omit $k$ in $\al|_{\supp(\be'),  k}$.
\begin{example}\label{preex}
Let $\al, \be \in S_6$, where $\al=(1\;2\;3)(4\;5\;6)$ and $\be=(1 \; 2\;3\;4)(5\;6)$. If $\be'=(1 \; 2\;3\;4)$, then $\al(1 \; 2\;3\;4)\al^{-1}=(2 \;3\;1\;5)$ and $\al|_{\supp(\be')}$ can be written as 
\[
\al|_{\supp(\be')}=\left(
\begin{array}{ccccccccc}
1 & 2 & 3  & 4 \\
2 & 3 & 1  & 5
 \end{array}
\right).
\]
Two ways of writing $\al|_{\supp(\be')}$ in block notation are
\[
\left(
\begin{array}{|cc|c|c|ccccc}
1 & 2 & 3  & 4 \\
2 & 3 & 1  & 5
 \end{array}
\right), \hspace{0.3cm}
\left(
\begin{array}{|c|c|c|c|ccccc}
1 & 2 & 3  & 4 \\
2 & 3 & 1  & 5
 \end{array}
\right),
\]
where the vertical lines denote the limits of the blocks. If $\be'=(2\;3\;4\;1)$ then block notation of $\al_{\supp(\be')}$ is
\[
\left(
\begin{array}{|c|c|c|c|ccccc}
 2 & 3  & 4 & 1\\
 3 & 1  & 5 & 2
 \end{array}
\right).
\]
\end{example}
\section{Formulas for $c(k, \be)$ for some cycle types of $\be$}\label{someformulas}

In this section, we show some formulas for $c(k, \be)$ when $\be$ is either a $3$-cycle, a $4$-cycle, or an $(n-1)$-cycle. The following observation is easy to prove directly or as a consequence of Theorem 3.8 in~\cite{morriv}.  
\begin{observation}\label{fixedp}
Let $\be$ be any permutation. 
\begin{enumerate}
\item If $x \in \fix(\be)$, then $\al\be(x)=\be\al(x)$ if and only if $\al(x) \in \fix(\be)$.
\item If $x \in \supp(\be)$ and $\al(x) \in \fix(\be)$, then  $\al\be(x) \neq \be\al(x)$.
\end{enumerate}
\end{observation}

\subsection{The case of $m$-cycles}

Let $\be_m=(b_1\dots b_m) \in S_n$ such that $\fix(\be_m)=\{f_1, \dots, f_{n-m}\}$. By Theorem 3.8 in \cite{morriv}, and without loss of generality, we have that any permutation $\al$ that $k$-commutes with $\be_m$, has the following block notation

\begin{equation}\label{alpha-arr}
\left(
\begin{array}{cccccccc}
B_1  & \dots  &B_{k_2}   \\
J_1 & \dots  & J_{k_2} 
 \end{array}\right)
 \left( \begin{array}{cccccccc}
f_1 \\
b_1' 
 \end{array} \right) \dots 
 \left( \begin{array}{cccccccc}
f_{k_1} \\
b_{k_1}' 
 \end{array} \right)
 \left( \begin{array}{cccccccc}
f_{k_1+1} \\
f_1' 
 \end{array} \right) \dots \left( \begin{array}{cccccccc}
f_{n-m} \\
f_{n-m-k_1}' 
 \end{array} \right),
\end{equation}
where: 1) $(B_1\dots B_{k_2})=(b_1 \dots b_m)$; 2) $f_i$ and $f_j'$ are fixed points of $\be_m$ and $b_l$ belongs to $\supp(\be_m)$, for every $i$, $j$ and $l$; 4) $J_i$ is either a block in $(b_1 \dots b_m)$ or a block in a $1$-cycle of $\be_m$, for $1\leq i \leq k_2$; 5) $J_iJ_{i+1 \pmod {k_{2}}}$ is not a block in $(b_1 \dots b_m)$, for $1\leq i \leq k_2$. 
 
Computing a general formula for $c(k, \be_m)$ for every $m$ and every $k$ seems to be a difficult task because we need to consider all solutions of the equation $k=k_1+k_2$ subject to the restriction $0 \leq k_1 \leq k_2$. Furthermore, for every one of those solutions we need to consider all possible ways to write $\al$ as in~(\ref{alpha-arr}). However, we have obtained some formulas for some choices of $k$ and $m$.  

\begin{proposition}\label{final}
Let $n, m \in \nn$, with $m \geq 2$ and $n\geq 2m-1$. If $\be_m \in S_n$ is an $m$-cycle, then
\begin{enumerate}
\item $c(2m, \be_m) = m!(n-m)!\binom{n-m}{m}$.
\item $c(2m-1, \be_m) =   m(n-m)!m!\binom{n-m}{m-1}$.
\end{enumerate}

\end{proposition}
\begin{proof}

Let $\be_m=(b_1\dots b_m) \in S_n$ such that $\fix(\be_m)=\{f_1, \dots, f_{n-m}\}$. Let $\al$ be any permutation that does not commute with $\be_m$ on exactly $k_1$ (resp. $k_2$) points in $\fix(\be_m)$ (resp. $\supp(\be_m)$). 

 By Observation~\ref{fixedp} we have that if $k=2m$ (or $k=2m-1$) then $k_1=m$ (or $k_1=m-1$). The result is obtained by constructing  all permutations that $k$-commute with $\be_m$. 

{\it Proof of part 1.}  We have that $k_1=k_2=m$ and hence $\al|_{\supp(\be_m)}$ should be a bijection from $\supp(\be_m)$ to a subset $\BB' \subseteq \fix(\be_m)$, with $|\BB'|=m$. There are $\binom{n-m}{m}$ ways to choose $\BB'$ and there are $m!$ bijections from $\supp(\be_m)$ onto  $\BB' $. Now we construct $\al|_{\fix(\be_m)}$. First select a set $\BB=\{f_{i_1}, \dots, f_{i_{m}}\} \subseteq	 \fix(\be_m)$ whose elements will be b.c.p. of $\al$ and $\be_m$ (in $\binom{n-m}{m}$ ways). Then we construct a bijection from $\BB$ onto $\supp(\be_m)$ (in $m!$ ways). Finally, we construct a  bijection from $\fix(\be_m) \setminus \BB$ onto $\fix(\be_m) \setminus \BB'$ (in $(n-2m)!$ ways). Therefore, we have
\[
c(2m, \be_m)=\binom{n-m}{m}^2(m!)^2(n-2m)!=m!(n-m)!\binom{n-m}{m}.
\]
{\it Proof of part 2.} For this case, $k_1=m-1$ and $k_2=m$.  We have that exactly $m-1$  points, in the support of $\be_m$, should be the images under $\al$ of  exactly $m-1$ fixed points. There are $\binom{n-m}{m-1}$ ways to select a subset $\BB \subseteq \fix(\be_m)$, with $|\BB|=m-1$, that will be the b.c.p. of $\al$ and $\be_m$. There are $m$ ways to select a set $\SS' \subset \supp(\be_m)$, $|\SS'|=m-1$, that will be the range of $\al|_{\BB}$ and there are $(m-1)!$ bijections from $\BB$ onto $\SS'$. Let $\{y'\}=\supp(\be_m) \setminus \SS'$. There are $\binom{n-m}{m-1}$ ways to select a set $\BB' \subseteq \fix(\be)$ of $m-1$ fixed points. By Theorem 3.8 in  \cite{morriv}, any bijection from $\supp(\be_m)$ onto $\BB' \cup \{y'\}$ produces $m$ b.c.p. in  $\supp(\be_m)$. There are $m!$ such bijections. Finally, the $(n-2m+1)!$ bijections from $\fix(\be_m) \setminus \BB$ onto $\fix(\be_m) \setminus \BB'$ produce only  g.c.p. as desired. Therefore, we have
\[
c(2m-1, \be_m)=\binom{n-m}{m-1}^2m(m-1)!m!(n-2m+1)!=m(n-m)!m!\binom{n-m}{m-1}.
\] 
\end{proof}

The formula for $c(2m, \be_m)$ can be written in terms of sequences in the OEIS database.  
\begin{eqnarray*}
c(2m, \be_m) &=& \binom{n-m}{m} \times m!(n-m)!\\
&=& A052553(n, m) \times A098361(n, m).
\end{eqnarray*}

\subsection{The case of $3$ and $4$-cycles}

We present formulas for $c(k, \be)$ when $\be$ is either a $3$-cycle or $4$-cycle.
\begin{theorem}\label{the-3cycle}
If $\be_3 \in S_n$ is a $3$-cycle, then
\begin{enumerate}
\item $c(0, \be_3)=3(n-3)!$, $n\geq 3$.
\item $c(3, \be_3)=(3(n-3)+1)3(n-3)!$, $n\geq 3$.
\item $c(4, \be_3)=3(n-3)3(n-3)!$, $n\geq 4$.
\item $c(5, \be_3)=6\binom{n-3}{2}3(n-3)!$, $n\geq 5$.
\item $c(6, \be_3)=2\binom{n-3}{3} 3(n-3)!$, $n\geq 6$.
\item $c(k, \be_3)=0$, for $7 \leq k \leq n$.
\end{enumerate} 
\end{theorem}
\begin{proof}
The case $c(k, \be_3)$, for $k \geq 7$, follows from Proposition 6.1 in \cite{morriv}. The cases $c(3, \be)$ and  $c(4, \be)$ follow from Theorems 5.1 and 5.2 in \cite{morriv}, respectively. The cases $c(5, \be)$ and $c(6, \be)$ follow from Proposition~\ref{final}.
\end{proof}

\begin{theorem}\label{the-4cycle}
If $\be_4 \in S_n$ is a $4$-cycle, then 

\begin{enumerate}
\item $c(0, \be_4) = 4(n-4)!$, $n \geq 4$.
\item $c(3, \be_4) = (4n-12)4(n-4)!$, $n \geq 4$.
\item $c(4, \be_4) = \left(1+8(n-4)\right)4(n-4)!$, $n \geq 4$.
\item $c(5, \be_4) = \left(12(n-4)+8\binom{n-4}{2}\right)4(n-4)!$, $n \geq 5$.
\item $c(6, \be_4)=\left(14n^2-126n+280\right)4(n-4)!$, $n \geq 6$.
\item $c(7, \be_4)= 24\binom{n-4}{3}  4(n-4)!, n\geq 7$.
\item $c(8, \be_4) = 6 \binom{n-4}{4}4(n-4)!$, $n \geq 8$.
\item $c(k, \be_4)=0$, for $9 \leq k \leq n$.
\end{enumerate}

\end{theorem}
\begin{proof}
The case $c(k, \be_4)$, for $k \geq 9$, follows from Proposition 6.1 in \cite{morriv}. The cases $k=3$ and $k=4$ follow from Theorems~5.1 and 5.2  in \cite{morriv}, respectively. The cases $k=7$ and $k=8$ follow from Proposition~\ref{final}.

{\it Proof of the case $k=5$}. 

Let $\al$ be a permutation that $5$-commute with $\be_4=(a_1 a_2 a_3 a_4)$. In this proof we use the notation $[s, f]$ to indicate that $\be_4$ has $s$ (resp. $f$) b.c.p. of $\al$ and $\be_4$ in $\supp(\be_4)$ (resp. $\fix(\be_4)$). When $k=5$, it is easy to see that the unique options are either $[4, 1]$ or $[3, 2]$.

{\it Subcase $[4, 1]$}. For this case, and without loss of generality, any permutation (in block notation) that $5$-commute with $\be_4$ are seen as  

\begin{equation*}
\left(
\begin{array}{cccccccc}
 a_1 & a_2 & a_3 & a_4   \\
f_1' & a_{i_1} & a_{i_2} & a_{i_3} 
 \end{array}\right)
 \left( \begin{array}{cccccccc}
f_1 \\
a_{i_4} 
 \end{array} \right)  
 \left( \begin{array}{cccccccc}
f_{2} \\
f_{2}' 
 \end{array} \right) \dots
  \left( \begin{array}{cccccccc}
f_{n-m} \\
f_{n-m}' 
 \end{array} \right),
\end{equation*}

where $f_j, f_j' \in \fix(\be_4)$, for $j \in \{1, \dots, n-m\}$, and $a_{i_j}  a_{i_{j+1}}$ is not a block in $(a_1 \dots a_4)$, for $j \in \{1, 2\}$.  There are $n-4$ ways to select the point $f_1$ and there are $4$ ways to select the point $a_{i_4}$. There are $4$ ways to select point $a_1 \in \supp(\be_4)$ and there are $n-4$ ways to select the fixed point $f_1'$. It is easy to check that once we have selected one point $z'$ from $\{a_1, a_2, a_3, a_4\} \setminus \{a_{i_4}\}$ as the image under $\al$ of one point $z$ in  $\{a_2, a_3, a_4\}$, the images under $\al$ of the points in $\{a_2, a_3, a_4\} \setminus \{z\}$ are uniquely determined, i.e., after the selection of $z$ (in $3$ ways) there are unique bijection from  $\{a_2, a_3, a_4\}$ onto $\supp(\be_4)\setminus \{a_{i_4}\}$ with the desired properties (to have four b.c.p. in $\supp(\be)$). Finally, there are $(n-5)!$ bijections from $\fix(\be_4)\setminus \{f_1\}$ onto $\fix(\be_4) \setminus \{f_1'\}$. Therefore, for this case we have $12(n-4)4(n-4)!$ permutations.

{\it Subcase $[3, 2]$}. For this case, and without loss of generality, any permutation (in block notation) that $5$-commute with $\be_4$ are seen as 

\begin{equation*}
\left(
\begin{array}{cccccccc}
 a_1 & a_2 & a_3 & a_4   \\
f_1' & f_2' & a_{i_1} & a_{i_2} 
 \end{array}\right)
 \left( \begin{array}{cccccccc}
f_1 \\
a_{i_3} 
 \end{array} \right)  
 \left( \begin{array}{cccccccc}
f_{2} \\
 a_{i_4}  
 \end{array} \right) 
 \left( \begin{array}{cccccccc}
f_{3} \\
 f_{3}'  
 \end{array} \right) 
 \dots
  \left( \begin{array}{cccccccc}
f_{n-m} \\
f_{n-m}' 
 \end{array} \right),
\end{equation*}

where $f_j, f_j' \in \fix(\be_4)$,  for $j \in \{1, \dots, n-m\}$, and $a_{i_1}  a_{i_{2}}$ is a block in $(a_1 \dots a_4)$. Notice that either $a_{i_3}a_{i_4}$ or $a_{i_4}a_{i_3}$ is a block in $(a_1 \dots a_4)$.
There are $\binom{n-4}{2}$ ways to select the subset $\{f_1, f_2\} \subseteq \fix(\be_4)$. There are $4$ ways to select the subset $\{x, y\} \subset \{a_1, a_2, a_3, a_4\}$ that will satisfy $\al(\{f_1, f_2\})=\{x, y\}=\{a_{i_3}, a_{i_4}\}$ (once we select a point, say $x$, the second is uniquely determined). There are $2$ bijections from $\{f_1, f_2\}$ onto $\{x, y\}$. There are $4$ ways to select $a_1$ from $\sup(\be_4)$. There are  $\binom{n-4}{2}$ ways to select the set $\{x', y'\} \subseteq \fix(\be_4)$  that will satisfy $\al(\{a_1, a_2\})=\{x', y'\}=\{f_1', f_2'\}$. There are $2$ bijections from $\{a_1, a_2\}$ onto $\{x', y'\}$. Finally there are $(n-6)!$ bijections from $\fix(\be_4) \setminus \{f_1, f_2\}$ onto $\fix(\be_4) \setminus \{f_1', f_2'\}$. Therefore, for this case, we have

\[
8\binom{n-4}{2}4(n-4)!,
\]
permutations $\al$ that $5$-commute with $\be_4$. Therefore
\[
c(5, \be_4)=12(n-4)4(n-4)!+8\binom{n-4}{2}4(n-4)!
\]
Finally, as
\[
c(6, \be_4)=n!-\left(c(7, \be_4)+c(8, \be_4)+\sum_{k=0}^5c(k, \be_4)\right), 
\] 
the result for $c(6, \be_4)$ follows by direct calculation.
\end{proof}

\subsection{The case of $(n-1)$-cycles}
Let $\pi=p_1 \dots p_n$ be a permutation of $[n]$ in its one line notation. A substring $p_i p_{i+1}$, with $1\leq i \leq n-1$, is called a {\it succession} of $\pi$ if $p_{i+1}=p_i+1$. Let $S(n)$ denote the number of permutations in $S_n$ without a succession.  In~\cite[Sec. 5.4]{chara} there are some formulas for $S(n)$. For example
\[
S(n)=(n-1)!\sum_{k=0}^{n-1}(-1)^k\frac{n-k}{k!}, n \geq 1,
\]
and the recursive formula
\[
S(n)=(n-1)S(n-1)+(n-2)S(n-2), \text{ for } n \geq 3,  \text{and } S(1) = S(2) =1.
\]
The sequence A000255~\cite{slo3} satisfies the recurrence relation
\[
A000255(n)=nA000255(n-1)+(n-1)A000255(n-2), \text{ for } n \geq 2,
\]
and $A000255(0)=1$ and $A000255(1)=1$. 

It is easy to see that 
$S(n)=A000255(n-1)$, for $n \geq 1$. 
Let $C(k)$ be the number of cyclic permutations of $\{1, \dots , k\}$ with no $i \mapsto i+1\bmod k$ (see \cite[Sec. 5.5]{chara}). The number $C(k)$ is sequence A000757 in \cite{oeis}. 
\begin{theorem}\label{the-n-1c}
Let $n \geq 4$ be an integer. If $\be_{n-1} \in S_n$ is any $(n-1)$-cycle, then 
\[
c(k, \be_{n-1})=(n-1)\binom{n-1}{k}C(k)+(n-1)^2\binom{n-3}{k-3}S(k-2). 
\]

\end{theorem}
\begin{proof}

Let $\be_{n-1}=(b_1\dots b_{n-1})(b_n)$ be an $(n-1)$-cycle in $S_n$. In this proof we write $\be$ instead of $\be_{n-1}$. Let $\CC=\{\al \in C(k, \be) \colon~\al(b_n)=b_n\}$. Notice that $\al \in \CC$ if and only if $\al\be(b_n)=\be\al(b_n)$, and this fact implies that for any $\al \in \CC$, $\al|_{\supp(\be)}$ is a permutation of $\{b_1, ..., b_{n-1}\}$ that $k$-commutes with the cyclic permutation $(b_1\dots b_{n-1})$. By Theorem~4.13 in \cite{morriv}, we have that $|\CC|=(n-1)\binom{n-1}{k}C(k)$. 

Let $\overline{\CC} =C(k, \be) \setminus \CC$. We calculate $|\overline{\CC} |$ by constructing all permutations in $\overline{\CC}$, i.e., permutations $\al$ such that $\al(b_n) \neq b_n$ and that do not commute with $\be$ on exactly $k-1$ points in $\supp(\be)$. Let $x=\al(b_n)$. First, it is easy to see that $\al \in \overline{\CC}$ if and only if $\al(b_n) \in \supp(\be)$. Therefore, Theorem 3.8 in \cite{morriv} implies that $\al$ restricted to $\supp(\be)$ can be written as

\begin{equation}\label{eqalpha-1}
\al|_{\supp(\be)}=\left(
\begin{array}{ccccccccc}
B_1 & \dots &B_j & b_r & B_{j+1} &\dots &B_{k-2}\\

B_{i_1}' & \dots & B_{i_j}' & b_n &\; B_{i_{j+1}}' & \dots &B_{i_{k-2}}' 
 \end{array}
 \right),
\end{equation}
where
\begin{enumerate}
\item the set $\{B_1, \dots, B_j,  b_r, B_{j+1}, \dots, B_{k-2}\}$ is a block partition of $(b_1  \dots b_{n-1})$ such that
\[
(b_1  \dots b_{n-1})=(B_1  \dots B_j b_rB_{j+1} \dots  B_{k-2});
\]
\item the string $B_{i_1}'  \dots  B_{i_j}'  \; B_{i_{j+1}}'  \dots B_{i_{k-2}}'$ is a block permutation of $B_1'  \ldots  B_{k-2}'$, where the set $\{B_1',   \ldots,  B_{k-2}', x\}$ is a block partition of $(b_1  \dots b_{n-1})$, with $|B_{i_j}'|=|B_j|$, for $1\leq j \leq k-2$, and such that
\begin{equation}\label{blocks-x}
(b_1  \dots b_{n-1})=(B_1' \dots B_{k-2}'x);
\end{equation}

\item $B_{i_s}'B_{i_{s+1}}'$ is not a block in $(b_1 \dots b_{n-1})$ for $s \in \{1, ..., k-3\} \setminus \{j\}$, and $B'_{i_j}B'_{i_{j+1}}$ may or not may be a block in $(b_1 \dots b_{n-1})$. 
\end{enumerate}

Now we count the number of ways to construct $\al \in \overline{\CC}$ as in~(\ref{eqalpha-1}). There are $(n-1)$ ways to choose $x \in \supp(\be)$ such that $x=\al(b_n)$.  There are $\binom{n-3}{k-3}$ ways to choose the block partition $\{B_1',   \ldots,  B_{k-2}', x\}$ of $(b_1 \dots b_{n-1})$. Indeed, we only need to choose the first element of $k-3$ blocks between $n-3$ points in $\{b_1, \dots, b_{n-1}\}$ because the corresponding first points of blocks $x$ and $B_1'$ in (\ref{blocks-x}) are uniquely determined ($x$ was already chosen). There are $n-1$ ways to select the first element of block $B_1$ in~(\ref{eqalpha-1}) and the rest of the blocks are uniquely determined by the lengths of blocks $B_1', \dots, B_{k-2}', x$. Then we have 
\[
(n-1)^2\binom{n-3}{k-3}R,
\]
ways to construct $\al|_{\supp(\be)}$ as in~(\ref{eqalpha-1}), where $R$ is the number of ways to construct the second row of the matrix in~(\ref{eqalpha-1}) in such a way that we have exactly $k-1$ b.c.p. of $\al$ and $\be$. Now, the matrix in~(\ref{eqalpha-1}) can be rewritten as 
\begin{equation}\label{eqalpha3}
\al|_{\supp(\be)}=\left(
\begin{array}{ccccccccc}
B_{j+1} &\dots &B_{k-2}& B_1 & \dots &B_j & b_r &\\
B_{i_{j+1}}' & \dots &B_{i_{k-2}}' &B_{i_1}' & \dots & B_{i_j}' & b_n
 \end{array}
 \right).
\end{equation}
Notice that $b_r$ is necessarily a b.c.p. of $\al$ and $\be$. Therefore, in order to obtain exactly $k-1$ b.c.p. we need for the blocks product 
\[ 
B_{i_{j+1}}'  \dots B_{i_{k-2}}' B_{i_1}'  \dots B_{i_j}'
\]
to not have a string of the form $B_r'B_{r+1}'$, for $1\leq r \leq k-3$ (in this way we obtain a b.c.p. per block), which is true if and only if $B_{i_{j+1}}'  \dots B_{i_{k-2}}' B_{i_1}'  \dots B_{i_j}'$ is equal to a block permutation $B_{\tau(1)}'  \dots B_{\tau(k-2)}'$ of $B_{1}'  \dots B_{k-2}'$, where $\tau$ is a permutation without a succession. As there are $S(k-2)$ such permutations we have that $R=S(k-2)$ and the result follows.  
\end{proof}
Previously, in joint work with Rutilo Moreno (that is part of his PhD Thesis~\cite{tesisrm}), we obtained the formula  
\[
c(k, \be_{n-1})=(n-1)\binom{n-1}{k}C(k)+(n-1)^2\binom{n-3}{k-3}T(k), \text{for } n \geq 4.
\]
where 
\[
T(k)=(k-2)C(k-2)+(2k-5)C(k-3)+(k-3)C(k-4), \text{for } k \geq 4.
\]
Using this formula and Theorem~\ref{the-n-1c}, we have the following relation between sequences A000255 and A000757, for $k \geq 4$.
{\small
\[
A000255(k-3)=(k-2) A000757(k-2) + (2k-5)A000757(k-3) + (k-3)A000757(k-4).
\]
}
When $k=3$ we have
\begin{eqnarray*}
c(3, \be_{n-1})&=&(n-1)\binom{n-1}{3} + (n-1)^2, \\
&=&(n-1) \times \left(\binom{n-1}{3} + n-1\right),\\
&=&(n-1) \times A000125(n-2), \text{ for }n\geq 3,
\end{eqnarray*}
where 
\[
A000125(m)=\binom{m+1}{3}+m+1, \text{ for }m \geq 0, 
\]
is the formula for the Cake numbers. We do not find sequences in the OEIS database that correspond to $c(\be_{n-1}, k)$ when $4 \leq k \leq 10$ and it is possible that no such relations existed until now for $k \geq 11$. 

By direct calculation we obtain 
\begin{eqnarray*}
c(\leq 3, \be_{n-1})&=&(n-1) \times \left(\binom{n-1}{3} + n\right).\\
\end{eqnarray*}
The sequence $A011826(m)$ is equal to $\binom{m}{3} +(m+1)$, for $1 \leq m \leq 1000$, as was noted Layman in~\cite{linu}. Then we have that
\[
c(\leq 3, \be_{n-1})=(n-1) \times A011826(n-1),\text{ for }  2 \leq n \leq 1000.
\]
\subsection{Other cases}

We obtain formulas for $c(n, \be)$ and $c(n-1, \be)$ for some special choices of $\be \in S_n$.    

\begin{proposition}\label{special}
If $\be \in S_{2m}$ is a permutation with exactly $m$ fixed points, then 
\begin{enumerate}
\item $c(2m, \be)=(m!)^2$.
\item $c(2m-1, \be)=m^2(m!)^2$.
\end{enumerate}
\end{proposition}
\begin{proof}
{\it Part 1).} Let $\al$ be any permutation that $2m$-commutes with $\be \in S_{2m}$. By Observation~\ref{fixedp}, $\al\left(\fix(\be)\right) \subseteq \supp(\be)$. As $|\fix(\be)|=|\supp(\be)|=m$ then $\al(\fix(\be))$ is equal to $\supp(\be)$, which implies that $\al(\supp(\be))=\fix(\be)$. Then $\al|_{\fix(\be)}$ (resp. $\al|_{\supp(\be)}$) is any bijection from $\fix(\be)$ (resp. $\supp(\be)$) onto $\supp(\be)$ (resp. $\fix(\be)$). Therefore we have $(m!)^2$ ways to construct $\al$.  \\
{\it Proof of part  2).} Let $x$ be the unique g.c.p. of $\al$ and $\be$. It is easy to see that $x$ and $\al(x)$ belongs to $\fix(\be)$.  
Then, $\al|_{\fix(\be)}$ is any bijection from $\fix(\be)$ onto $\supp(\be) \setminus \{a\} \cup \{x'\}$, for some $a \in \supp(\be)$ and $x' \in \fix(\be)$  (there are $m!$ such bijections), and $\al|_{\supp(\be)}$ is any bijection from $\supp(\be)$ onto $\fix(\be)\setminus \{x'\} \cup \{a\}$ (there are $m!$ such bijections). As we have $m$ ways to select $x'$ and $m$ ways to select $a$ then $c(2m-1, \be)=m^2 (m!)^2$. 
\end{proof}

For $\be$ as in the previous proposition we have
\[
c(2m, \be)=(m!)^2=A001044(m), m \geq 0,
\]
and
\begin{eqnarray*}
c(2m-1, \be)&=&m^2 \times (m!)^2=A000290(m) \times A001044(m),\\
&=&m \times m(m!)^2 = m \times A084915(m).
\end{eqnarray*}
From this we obtain the following identity 
\[
m \times A084915(m) = A000290(m) \times A001044(m).
\]

\begin{proposition}
If $\be \in S_{2m-1}$ is a permutation with exactly $m-1$ fixed points, then $c(2m-1, \be)=(m!)^2$. 
\end{proposition}
\begin{proof}
As all the fixed points are b.c.p. of $\al$ and $\be$, then $\al\left(\fix(\be)\right)$ is a subset of $\supp(\be)$ and $\al\left(\supp(\be)\right)=\fix(\be) \cup \{a\}$, for some $a \in \supp(\be)$. Then, $\al|_{\fix(\be)}$ should be any bijection from $\fix(\be)$ onto a subset $B$ of $\supp(\be)$, with $|B|=m-1$, and $\al|_{\supp(\be)}$ should be any bijection from $\supp(\be)$ onto $\fix(\be) \cup \left(\supp(\be) \setminus B\right)$. There are $\binom{m}{m-1}(m-1)!=m(m-1)!$ ways to construct $\al|_{\fix(\be)}$ and there are $m!$ ways to construct $\al|_{\supp(\be)}$. Then we have that $c(2m-1,\be)=m(m-1)!m!=(m!)^2$. 
\end{proof}

\subsection{Upper bound for $c(\leq k, \be)$ when $\be$ is an $n$-cycle}
In the following theorem we use the convention that $(-1)!=1$. 

\begin{theorem}\label{the-n-cycle}
If $\be$ is an $n$-cycle, then
\[
c(\leq k, \be) \leq n\binom{n}{k}(k-1)!-n\binom{n}{k}+n,
\]
with equality for $0 \leq k \leq 3$ and $k=n$.
\end{theorem}
\begin{proof}
Let $\be=(b_1 \dots b_n)$. First, we select a block partition $\{B_1, \dots, B_k\}$ of $b_1 \dots b_n$ (in $\binom{n}{k}$ ways), such that $B_1  \dots  B_k=b_1 \dots b_n$. Next, we construct a permutation $\al$ in block notation as \[
\al=\left(\begin{array}{cccc}
P_1  &\dots &P_k   \\
 B_{i_1}   & \dots & B_{i_k}  
 \end{array} \right),
\]  
where $(P_1 \dots P_k)=(b_1\dots b_n)$, $B_{i_1}  \dots  B_{i_k}$ is any block permutation of $B_1  \dots  B_k$, and $|P_j|=|B_{i_j}|$, for $1\leq j\leq k$. We have $n$ ways to select the first element in block $P_1$ and the rest of the blocks in first row are uniquely determined by the lengths of blocks $B_i$. For each selection of the partition $\{B_1, \dots, B_k\}$, there are $(k-1)!$ ways to arrange the blocks $B_1, \dots, B_k$ in the second row of this matrix. Indeed, for each selection of partition $\{B_1,  \dots, B_k\}$, the $k$ cyclic permutation of $B_1\dots B_k$ will give $k$ repeated permutations. Thus, there are at most $n\binom{n}{k}(k-1)!$ permutations that $(\leq k)$-commute with $\be$. We can reduce this bound a little more. Notice that for every one of the $n$ different possibilities  for $P_1\dots P_k$ in the first row of the matrix, the vector $\langle b_1, \dots, b_n \rangle$ appears $\binom{n}{k}$ times in the second row (there are $\binom{n}{k}$ block partitions $\{B_1, \dots, B_k\}$ of $b_1 \dots b_n$ such that $B_1  \dots  B_k=b_1 \dots b_n$, and every one of this partitions appear once in the second row).  Therefore, we have 
\begin{eqnarray*}
 c(\leq k, \be) \leq n\binom{n}{k}(k-1)!-n\binom{n}{k}+n,
\end{eqnarray*}
The fact that equality is reached when $0\leq k \leq 3$ and $k=n$ follows by direct calculation.  
\end{proof}

\section{Relations with integer sequences}\label{relations}
In this section, we write some formulas for $c(k, \be)$ and $c(\leq k, \be)$ in terms of sequences in the OEIS database. These formulas are obtained by simple inspection or with the help of a computer algebra system and an exhaustive search on the OEIS database. We use the following notation. We write formula $a \times b$ as A$i \times$A$j$, to mean that $a$ (or $b$) is the formula for sequence A$i$ (or A$j$) in the OEIS database.  
\subsection{Transpositions}\label{trans-bij}

In \cite{morriv} we showed the following result.
\begin{proposition}\label{transpo}
If $\be_2 \in S_n$ is a transposition, then
\begin{enumerate}
\item $c(0, \be_2)=2(n-2)!$, $n \geq 2$.
\item $c(3,\be_2)=4(n-2)(n-2)!$, $n \geq 3$.
\item $c(4,\be_2)=(n-2)(n-3)(n-2)!$, $n \geq 4$.
\item $c(k, \be_2)=0$, $5\leq k\leq n$.
\end{enumerate}
\end{proposition}

In \cite{morriv} it was noted that  $c(0, \be_2)$, $c(3, \be_2)$, and  $c(4, \be_2)$ coincide with the number of permutations in  $S_n$, $n \geq 2$, having exactly $2$, $3$, and $4$ points, respectively,  on the boundary of their bounding square \cite{deutsch} (sequences A208529,  A208528, and A098916, respectively). Here we provide an explicit relationship between these numbers. The following definition is taken from problem 1861 in \cite{deutsch}. A permutation $\al \in S_n$
 can be represented in the plane by the set of $n$ points $P_\al=\{(i, \al(i))\colon ~1\leq i \leq n\}$. The {\it bounding square} of $P_\al$ is the smallest square bounding $P_\al$. In other words, the bounding square can be described as the square with sides parallel to the coordinate axis containing $(1,1)$ and $(n,n)$ (see \cite{nacin}). 
\begin{proposition}
\begin{enumerate}
\item[]
\item A permutation $\al$ has only two points on the boundary of their bounding square if and only if $\al$ commutes with transposition $(1, n)$.
\item A permutation $\al$ has only $m$ points on the boundary of their bounding square if and only if $\al$ $m$-commutes with transposition $(1, n)$, for $m \in \{3, 4\}$.
\end{enumerate}
\end{proposition}
\begin{proof}
In this proof we use some paragraphs in \cite{deutsch}.

{\em Proof of 1:} Permutation $\al$ commutes with $(1, n)$ if and only if either $\al(1)=1$ and $\al(n)=n$, or $\al(1)=n$  and $\al(n)=1$, i.e., if and only if $P_\al$ contains both $(1, 1)$ and $(n, n)$, or both $(1, n)$ and $(1, n)$, which is true if and only if $P_\al$ has exactly two points on the boundary of its bounding square, see  \cite{deutsch}. 

{\em Proof of 2:} Permutation $\al$ $4$-commutes with $(1, n)$ if and only if $\al(1) \not\in\{1, n\}$ and $\al(n) \not\in\{1, n\}$, i.e., if and only if $P_\al$ does not contain any of the points $(1, 1)$, $(n, n)$, $(1, n)$, and $(1, n)$, which is true if and only if $P_\al$ has exactly four points on the boundary of its bounding square, see  \cite{deutsch}. The case $m=3$ follows because both $\{C(0, \be), C(3, \be), C(4, \be)\}$ and the collection of sets $B_m$ of permutations $\al$ that have $m$ points on the boundary of the bonding square of $P_\al$, for $m \in \{2, 3, 4\}$,  are partitions of $S_n$.   
\end{proof}

By direct calculation we obtain 
\begin{eqnarray*}
c(\leq 3, \be_2)&=&2(n-2)!+4(n-2)(n-2)!,\\
&=& 2 \times (2(n-2)+1)(n - 2)!,\\
&=& 2 \times A007680(n-2), n \geq 2.\\
c(\leq 4, \be_2)&=&n!,\\
&=&A000142(n), n\geq 2.
\end{eqnarray*}

From these  formulas we obtain the following relation that was announced by the author in the FORMULA section of sequence A007680.
\begin{eqnarray*}
A007680(n-2)&=&\frac{A208529(n)+A208528(n)}{2}, n\geq 2.
\end{eqnarray*}
We can also obtain the following relations.
\begin{eqnarray*}
A000142(n) &=& A208529(n)+A208528(n)+A098916(n), n \geq 3.\\
A000142(n) &=&A098916(n)+2\times A007680(n-2), n \geq 3.
\end{eqnarray*}

\subsection{Some cases for $3$ and $4$-cycles}

Tables~\ref{table3c} and~\ref{table3cl} show formulas for $c(k, \be_3)$ (Theorem~\ref{the-3cycle}) and $c(\leq k, \be_3)$, respectively, in  terms of sequences in the OEIS database. The formulas for $c(\leq k, \be_3)=\sum_{i=0}^kc(k, \be_3)$ are obtained by direct calculation.  

\begin{table}[htdp]
\scriptsize
\begin{center}
\begin{tabular}{|c|l|l|l|l|}
\hline
k  & $c(k, \be_3)$ &   \\
\hline
0 & $3(n-3)!$ & $A052560(n-3)$, $n \geq 3$ \\ 
\hline
3 & $(3(n-3)+1) \times 3(n-3)!$ & $A016777(n-3) \times A052560(n-3)$, $n \geq 3$  \\
\hline
4 & $3(n-3) \times 3(n-3)!$, $n\geq 4$ & $A008585(n-3) \times A052560(n-3)$, $n \geq 3$ \\
\hline
5 & $3(n-4)(n-3) \times 3(n-3)!$, $n\geq 5$& $A028896(n-4) \times A052560(n-3)$, $n \geq 4$ \\
& $3(n-4) \times 3(n-3)(n-3)!$, $n\geq 5$ & $A008585(n-4) \times A083746(n-1), n \geq 5$\\
& & $A008585(n-4) \times A052673(n-3), n \geq 5$ \\
& $9(n-4) \times (n-3)(n-3)!, n \geq 5$ & $A008591(n-4) \times A001563(n-3), n \geq 5$ \\
& $3 \times 3 (n-4)  \times (n-3)(n-3)!$, $n\geq 5$ & $3 \times A008585(n-4) \times  A001563(n-3), n \geq 5$ \\
\hline
6 & $2\binom{n-3}{3} \times 3(n-3)!$, $n\geq 6$ & $A007290(n-3) \times A052560(n-3)$, $n \geq 6$ \\
&$6\binom{n-3}{3} \times (n-3)!$, $n\geq 6$ &  $A007531(n-3) \times A000142(n-3) , n\geq 6$\\
& $(n-3)(n-4) \times (n-5)(n-3)!, n \geq 6$ & $A002378(n-4) \times A052571(n-5), n \geq 6$\\
& $(n-5)(n-3) \times (n-3)! (n-4), n \geq 6$ & $A005563(n-5) \times A062119(n-3), n \geq 6$\\
& $(n-5)(n-4) \times (n-3) (n-3)!, n \geq 6$ & $A002378(n-5) \times A001563(n-3), n \geq 6$\\
& $(n-3)^2 (n-4) \times (n-4)! (n-5), n \geq 6$ & $A045991(n-3) \times A062119(n-4), n \geq 6$\\
& $(n-3)^2 (n-4)(n-5) \times (n-4)! , n \geq 6$ & $A047929(n-3) \times A000142(n-4), n \geq 6$\\
& $(n-3)^2 \times (n-4)(n-5) (n-4)! , n \geq 6$ & $A000290(n-3) \times A098916(n-2), n \geq 6$\\
& $(n-5)(n-3)^2 \times (n-4)(n-4)!, n \geq 6$ & $A152619(n-5) \times A001563(n-4), n \geq 6$\\
\hline
$\geq 7$& 0 & $A000004(n)$\\
\hline
\end{tabular}
\end{center}
\caption{Formulas for $c(k, \be_3)$ written in terms of sequences in the OEIS database.}
\label{table3c}
\end{table}
\begin{table}[htdp]
\scriptsize
\begin{center}

\begin{tabular}{|c|l|l|l|l|}
\hline
k  & \centering{$c(\leq k, \be_3)$} &   \\
\hline
0 & $3(n-3)!$ & $A052560(n-3)$, $n \geq 3$ \\ 
\hline
3 & $(3n-7) \times 3(n-3)!$ & $A016789(n-3) \times A052560(n-3)$, $n \geq 3$ \\

 & $(9n-21) \times (n-3)!$ & $A017233(n-3) \times A000142(n-3), n \geq 3$ \\
 \hline
4 & $\left(6(n-3)+2\right) \times 3(n-3)!$ & $A016933(n-3)  \times A052560(n-3)$, $n \geq 4$ \\
& $6 \times \left(3(n-3)+1\right)(n-3)!$ & $6 \times A082033(n-3)$ \\
& $\left(9(n-3)+3\right) \times 2((n-1)-2)!$ & $A017197(n-3) \times A208529(n-1), n \geq 4$\\
\hline
5 & $(20-15n+3n^2)3(n-3)!$ & $A077588(n-2)  \times A052560(n-3)$, $n \geq 5$\\
\hline
$\geq 6$ & $n!$ & $A000142(n)$, $n \geq 6$ \\
& $2\binom{n}{3} \times 3(n-3)!$ & $A007290(n)  \times A052560(n-3)$,  $n \geq 6$\\
& & \\
\hline
\end{tabular}
\end{center}
\caption{Formulas for $c(\leq k ,\be_3)$ written in terms of sequences in the OEIS database.}
\label{table3cl}
\end{table}

From these tables we obtain the following identities.
{\small
\begin{eqnarray*}
A016789(n) &=& A016777(n)+1, n \geq 0. (\text{see}~\cite{slo2}).\\
A016933(n) &=& A016789(n)+A008585(n), n \geq 0.\\
                    &=& A016777(n)+A008585(n)+1, n \geq 0. \\
A077588(n-2) &=& A016933(n-3)+A028896(n-4), n \geq 4.\\
 &=& A016789(n-3)+A008585(n-3)+A028896(n-4), n \geq 4.\\
&=& A016777(n-3)+A008585(n-3)+A028896(n-4)+1, n \geq 4. \\
A007290(n) &=& A077588(n-2)+A007290(n-3), n \geq 3.\\
&=& A016933(n-3)+A028896(n-4) +A007290(n-3), n \geq 4. \\
&=& A016789(n-3)+A008585(n-3)+A028896(n-4) +\\
&&+ \ A007290(n-3), n \geq 4.
\end{eqnarray*}
}
Tables~\ref{table4c} and~\ref{table4cl} show formulas for $c(k, \be_4)$ and $c(\leq k, \be_4)$, respectively, in  terms of sequences in the OEIS database. From these tables we deduce the following identities.
\begin{eqnarray*}
A052578(n)&=& A000142(n-1) \times \left(A016813(n)-1\right), n \geq 1.\\
A052578(n)&=& A000142(n-1) \times A008586(n), n \geq 1.\\
A008598(n)&=& 4 \times \left(A016813(n)-1\right), n \geq 0.\\
A016813(n) &=& A017593(n-1) - A017077(n-1), n \geq 1.\\
\end{eqnarray*}
We encourage the interested reader to obtain more identities from these tables.
\begin{table}[htdp]
\scriptsize
\begin{center}
\begin{tabular}{|c|l|l|}
\hline
k  & $c(k, \be_4)$ &   \\
\hline
0 & $4 \times (n-4)!$ & $4\times A000142(n-4), n \geq 4$\\ 
& $4  (n-4)!$ & $A052578(n-4), n \geq 5$\\
&  $2 \times 2 (n-4)!$ & $2 \times A208529(n-2), n \geq 4$\\
& & $2 \times A052849(n-4), n \geq 5$\\
\hline
3 & $4 \times 4(n-3)!$& $4 \times A052578(n-3), n \geq 4$\\
& $8 \times 2 (n-3)!$ & $8 \times A208529(n-1), n \geq 4$\\
& $16(n-3) \times (n-4)!$ & $A008598(n-3)\times A000142(n-4), n \geq 4$\\
 & $4(n-3) \times 4(n-4)!$ & $A008586(n-3) \times A052578(n-4), n \geq 5$\\
 \hline
4 & $\left(1+8(n-4)\right) \times 4  (n-4)!$ & $A017077(n-4) \times A052578(n-4), n \geq 5$\\
\hline
5 &$(2(n-3))^2-4) \times 4(n-4)!, n \geq 4$ & $A134582(n-3)\times A052578(n-4), n \geq 5$\\
\hline
6 & $14 \times (n-5)(n-4) \times 4  (n-4)!$ &  $14\times A002378(n-5)  \times A052578(n-4), n \geq 5$\\
 & $28 \times \binom{n-4}{2} \times 4(n-4)!$ & $28 \times A000217(n-5) \times A052578(n-4), n \geq 6$\\ 
 \hline
7 
& $24 \times \binom{n-4}{3}  \times 4  (n-4)!,$ & $24 \times A000292(n-6) \times A052578(n-4), n \geq 6$\\
& $4 \times 6 \binom{n-4}{3} \times  4(n-4)!$ & $4 \times A007531(n-4) \times A052578(n-4), n \geq 7$\\
& $3 \times 8 \binom{n-4}{3}\times  4(n-4)!$ & $3 \times A130809(n-4) \times A052578(n-4), n \geq 7$\\
\hline
8 & $24 \times \binom{n-4}{4}\times (n-4)!$ & $ 24 \times A000332(n-4) \times  A000142(n-4), n \geq 8$ \\
 & $24 \binom{n-4}{4} \times (n-4)!$  & $A052762(n-4) \times A000142(n-4), n \geq 8$ \\
&$6 \times \binom{n-4}{4} \times 4  (n-4)!$ & $ 6 \times A000332(n-4) \times A052578(n-4), n \geq 8$\\
& $6 \binom{n-4}{4} \times 4  (n-4)!$ & $A033487(n-7) \times A052578(n-4), n \geq 8$\\
\hline
$\geq 9$& 0 & $A000004(n)$\\
\hline
\end{tabular}
\end{center}
\caption{Formulas for $c(k, \be_4)$ written in terms of sequences in the OEIS database.}
\label{table4c}
\end{table}

\begin{table}[htdp]
\scriptsize
\begin{center}
\begin{tabular}{|c|l|l|}
\hline
k  & $c(\leq k, \be_4)$ &   \\
\hline

 $3$ & $4 \times \left(4(n-3)+1 \right) \times  (n-4)!$ & $4 \times A016813(n-3) \times  A000142(n-4)$, $n \geq 4$\\
& $\left(4(n-3)+1 \right) \times 4 (n-4)!$ & $A016813(n-3) \times A052578(n-4)$, $n \geq 5$\\
\hline
$4$ & $4 \times \left(12(n-4)+6\right) \times  (n-4)! $ & $4 \times  A017593(n-4)  \times A000142(n-4)$, $n \geq 4$\\
  & $8 \times \left(6(n-4)+3\right) \times (n-4)!$ & $8 \times A016945(n-4)  \times A000142(n-4)$, $n \geq 4$\\
  & $ \left(12(n-4)+6\right) \times 4(n-4)! $ & $ A017593(n-4)  \times A052578(n-4), n \geq 5$\\
 & $2 \times \left(6(n-4)+3\right) \times 4(n-4)!$ & $2 \times A016945(n-4)  \times A052578(n-4), n \geq 5$\\
\hline
$ 5$ & $8 \times \left(2(n-4)^2+10(n-4)+3\right)  \times (n-4)!$ & $8 \times A152813(n-4) \times A000142(n-4)$ , $n \geq 4$\\
& $2 \times \left(2(n-4)^2+10(n-4)+3\right) \times 4 (n-4)!$ & $2 \times A152813(n-4) \times A052578(n-4), n \geq 5$\\
\hline
$6$ & $ 24 \times \left( 3(n-4)^2+(n-4)+1\right) \times  (n-4)!$ & $24 \times A056108(n-4) \times  A000142(n-4)$ , $n \geq 4$\\
& $ 6 \times \left(3(n-4)^2+(n-4)+1\right) \times 4(n-4)!$ & $6 \times A056108(n-4) \times A052578(n-4), n \geq 5$\\
\hline
$7$ & $24 \times \left(2n^3-21n^2+79n-105 \right) \times (n-4)!$ & $24 \times A005894(n-4) \times A000142(n-4)$, $n \geq 4$\\
 & $6 \times \left(2n^3-21n^2+79n-105 \right) \times 4 (n-4)!$ & $6 \times A005894(n-4) \times A052578(n-4), n \geq 5$\\
\hline
$\geq $ & $n(n-1)(n-2)(n-3)/4 \times 4 \times (n-4)!$ & $A033487(n) \times 4 \times A000142(n-4)$, $n \geq 4$\\
8& $ n(n-1)(n-2)(n-3)/4 \times 4  (n-4)!$ & $A033487(n) \times A052578(n-4), n \geq 5$\\
\hline
\end{tabular}
\end{center}

\caption{Formulas for $c(\leq k ,\be_4)$ written in terms of sequences in the OEIS database.}
\label{table4cl}
\end{table}

\subsection{The case of $n$-cycles}
Let $C(k)=A000757(k)$. The following proposition was proved in \cite{morriv}.
\begin{proposition}
Let $n$ and $k$ be nonnegative integers, $0\leq k \leq n$. Let $\be_n \in S_n$ be an $n$-cycle. Then 
\[
c(k, \be_n)=n\binom{n}{k} C(k). 
\]
\end{proposition}

Let $C_{n, k}= n\binom{n}{k} C(k)$. The triangle $\{C_{n, k}\}$ is labeled now as sequence A233440 in~\cite{oeis}. 
 In Table~\ref{rela1} we show some formulas for $c(k, \be_n)$ written in terms of sequences in the OEIS database, for $0 \leq k \leq 16$. This table was announced by the author in \cite{rivera}. We searched for some values of $k \geq 17$, but no relationships with sequences in the OEIS database were found until now. 

\begin{table}[htdp]
\scriptsize
\begin{center}
\begin{tabular}{|c|l|}
\hline
$k$  & $c(k, \be_n)$ \\
\hline
0 & A001477$(n)$, $n \geq0$\\
\hline
1  & A000004$(n)$, $n \geq1$\\
\hline
2  & A000004$(n)$, $n \geq2$\\
\hline
3  & A004320$(n-2)\ = \ $A047929$(n)/6$, $n \geq3$\\
\hline
4 & A027764$(n-1)$, $n \geq4$\\
\hline
5  & A027765$(n-1) \times$A000757$(5)$, $n \geq5$  \\
\hline
6  & A027766$(n-1)\times$A000757$(6)$, $n \geq6$ \\
\hline
7 & A027767$(n-1)\times$A000757$(7)$, $n \geq7$ \\
\hline
8  & A027768$(n-1)\times$A000757$(8)$, $n \geq8$ \\
\hline
9 & A027769$(n-1)\times$A000757$(9)$, $n \geq9$ \\
\hline
10 & A027770$(n-1)\times$A000757$(10)$, $n \geq10$ \\
\hline
11 & A027771$(n-1)\times$A000757$(11)$, $n \geq11$ \\
\hline
12 & A027772$(n-1)\times$A000757$(12)$, $n \geq12$\\
\hline
13 & A027773$(n-1)\times$A000757$(13)$, $n \geq13$\\
\hline
14 & A027774$(n-1)\times$A000757$(14)$, $n \geq14$\\
\hline
15 & A027775$(n-1)\times$A000757$(15)$, $n \geq15$\\
\hline
16 & A027776$(n-1)\times$A000757$(16)$, $n \geq16$\\
\hline
\end{tabular}
\end{center}
\caption{Formulas for $c(k, \be_n)$, $0 \leq k \leq 16$, written in terms of sequences in the OEIS database.}
\label{rela1}
\end{table}

By direct computation we obtain the following formulas for  $c(\leq k, \be)$, for $k\in\{3, 4\}$. For the case of $k \in \{5, 6, 7\}$ no relationship with sequences in OEIS was found and it is possible that for the greatest values of $k$ no such relationship existed until now. 

\begin{eqnarray*}
 c(\leq 3, \be_n)&=&n\left(1+\binom{n}{3}\right) = n \times A050407(n+1), n \geq 0.\\
 c(\leq 4, \be_n)&=&n\left(1+\binom{n+1}{4}\right) = n \times A145126(n-2), n\geq 2.
\end{eqnarray*}

From this we can deduce the following identities 
\begin{eqnarray*}
n \times A050407(n+1)&=& A233440(n, 0) + A233440(n, 3), n \geq 0.\\
n \times A145126(n-2) &=&A233440(n, 0) + A233440(n, 3) + A233440(n, 4), n \geq 2.\\
\end{eqnarray*}

From these last relations and from the relations in Table~\ref{rela1} we have

\begin{eqnarray*}
A050407(n+1)&=&A004320(n-2)/n+1, n \geq 2.\\
A233440(n, 4)&=&n \times \left(A145126(n-2)-A050407(n+1)\right), n \geq 2.\\
A027764(n-1) &=& n \times \left(A145126(n-2)-A050407(n+1)\right), n\geq 4.
\end{eqnarray*}

\subsection{Some cases for $k=3$ and $k=4$}
R. Moreno and the author of this article presented in~\cite{morriv} formulas for $c(3, \be)$ and $c(4, \be)$, where $\be$ is any permutation. We use these results to obtain formulas for some type of permutations. Let $\be=(1 \dots  m)(m+1  \dots  2m) \in S_{2m}$. Then 
\begin{eqnarray*}
c(3, \be)&=&2m^2 \times 2 \binom{m}{3},\\
&=&A001105(m) \times A007290(m), m \geq 2.\\
c(4, \be)&=&2m^2 \times \left(2\binom{m}{4}+m\binom{m}{2}\right),\\
&=& 2m^2 \times \left(2\binom{m}{4}+ (m-1) m^2/2 \right),\\
&=& A001105(m) \times \left(A034827(m)+ A006002(m-1)\right), m \geq 2.
\end{eqnarray*}

Let $\be_m \in S_n$ be an $m$-cycle, for $n \geq 3$ and $1 < m \leq n$. By Theorem 5.1 in~\cite{morriv}, we have  
\[
c(3, \be_m)=m(n-m)!\left(\binom{m}{3}+m(n-m)\right). 
\]
In Table~\ref{table3cmcycle} we write formulas for $c(3, \be_m)$ for some values of $m$ in terms of sequences in the OEIS database.

\begin{table}[htdp]
\scriptsize
\begin{center}

\begin{tabular}{|c|l|l|}
\hline
$m$  & $c(3, \be_m)$ &   \\
\hline
$5$ & $25 \times \left((n-4)!+(n-5)!\right)$ & $25 \times A001048(n-4)!, n \geq 5$\\
& $5(n-3) \times 5 (n-5)!$ & $A008587(n-3) \times A052648(n-5)$, $n \geq 6$ \\ 
\hline
$6$ & $6 \times \left(6(n-3)+2\right) \times  (n-6)!$ & $6 \times A016933(n-3) \times  A000142(n-6)$, $n \geq 6$\\
\hline
$7$ & $7(n-2) \times 7  (n-7)!$ & $A008589(n-2) \times A062098(n-7)!, n \geq 8$\\
\hline
$8$ & $8(n-1) \times 8(n-8)!$ & $A008590(n-1) \times A159038(n-8), n \geq 9$\\
\hline
$9$ & $9 \times  (9n+3) \times (n-9)!$ & $9 \times  A017197(n) \times A000142(n-9), n \geq 9$\\
\hline 
$10$ & $10(n+2)\times 10  (n-10)! $& $A008592(n+2) \times A174183(n-10), n \geq 11$\\
\hline 
$11$ & $11 \times 11(n+4)\times  (n-11)!$ & $11 \times A008593(n+4)  \times A000142(n-11) , n \geq 11$\\
\hline 
$12$ & $12 \times  \left(12(n+6)+4\right) \times (n-12)!$ & $12 \times A017569(n+6)  \times  A000142(n-12) , n \geq 12$\\
\hline 
$13$ & $13 \times 13(n+9) \times (n-13)!$ & $13 \times  A008595(n+9)  \times  A000142(n-13) , n \geq 13 $\\
\hline
$14$ & $14 \times 14(n+12)\times  (n-14)!$ & $14 \times A008596(n+12)  \times  A000142(n-14) , n \geq 14 $\\
\hline
16 & $16 \times 16(n+19)\times  (n-16)!$ & $16 \times A008598(n+19)  \times  A000142(n-16) , n \geq 16 $\\
\hline
17 & $17 \times 17(n+23)\times  (n-17)!$ & $17 \times A008599(n+23)  \times  A000142(n-17) , n \geq 17 $\\
\hline
19 & $19 \times 19(n+32)\times  (n-19)!$ & $19 \times A008601(n+32)  \times  A000142(n-19) , n \geq 19 $\\
\hline
20 & $20 \times 20(n+37) \times (n-20)!$ & $20 \times  A008602(n+37) \times A000142(n-20), n \geq 20$\\
\hline
22 & $22 \times 22(n+48) \times (n-22)!$ & $22 \times A008604(n+48) \times  A000142(n-22), n \geq 22$\\
\hline
23 & $23 \times  23(n+54) \times (n-23)!$ & $23 \times A008605(n+54) \times  A000142(n-23), n \geq 23$\\
\hline
25 & $25 \times 25(n+67) \times  (n-25)!$ & $25 \times A008607(n+67) \times  A000142(n-25), n \geq 25$\\
\hline 
28 & $28 \times 28(n+89) \times  (n-28)!$ & $28 \times  A135628(n+89) \times A000142(n-28), n \geq 28$\\
\hline
29 & $29 \times 29(n+97) \times  (n-29)!$ & $29 \times A195819(n+97) \times  A000142(n-29), n \geq 29$\\
\hline
31 & $31\times 31(n+114) \times  (n-31)!$ & $31 \times A135631(n+114) \times  A000142(n-31), n \geq 31$\\
\hline
37 & $37 \times 37(n+173) \times  (n-37)!$ & $37 \times A085959(n+173) \times  A000142(n-37), n \geq 37$\\
\hline
\end{tabular}
\end{center}
\caption{Formulas for $c(3, \be_m)$ written in terms of sequences in the OEIS database.}
\label{table3cmcycle}
\end{table}

Let $\be_m \in S_n$ be an $m$-cycle, $4 \leq m \leq n$. By Theorem 5.2 in~\cite{morriv}, we have 
\[
c(4, \be_m)=m(n-m)!\left(\binom{m}{4}+m(m-2)(n-m)\right).
\]
In this case we found relationships with sequences in the OEIS database only for $4 \leq k \leq 7$ and these are shown in Table~\ref{table4cmcycle}.

\begin{table}[htdp]
\scriptsize
\begin{center}
\begin{tabular}{|c|l|l|l|l|}
\hline
$m$  & $c(4, \be_m)$ &   \\
\hline
5 & $25 \times \left(3(n-5) +1\right)\times (n-5)!$ & $25 \times A016777(n-5) \times  A000142(n-5), n \geq 5$\\
\hline
6 & $18 \times  \left(8(n-5)-3\right) \times (n-6)!$ & $18 \times A004770(n-5) \times  A000142(n-6), n \geq 6$\\
\hline
7 & $35 \times (n-6) \times 7 (n-7)!$ & $35 \times A000027(n-6) \times A062098(n-7)$, $n \geq 8$\\
& $245 \times  (n-6)!$ & $245 \times A000142(n-6), n \geq 7$\\
\hline
\end{tabular}
\end{center}
\caption{Formulas for $c(4, \be_m)$ for some $m$-cycles}
\label{table4cmcycle}
\end{table}

\begin{table}[htdp]
\scriptsize
\begin{center}

\begin{tabular}{|c|l|l|l}
\hline
$m$  &  $c(0, \be_m)$ & \\
\hline
$5$& $5(n-5)!$ & $A052648(n-5), n \geq 6$\\
\hline
$ 7$& $7(n-7)!$ &$A062098(n-7), n \geq 8$\\
\hline
$8$&  $8(n-8)!$ & $A159038(n-8), n \geq 9$\\
\hline
$10$& $10(n-10)!$ & $A174183(n-10), n \geq 11$\\
\hline
\end{tabular}
\end{center}
\caption{Formulas for the centralizer of some $m$-cycles.}
\label{centralizer}
\end{table}

\section{Final results and comments} \label{finalc}

We conclude this paper by considering the case of the size of the centralizer of some permutations. When $k=0$, $c(k, \be)$ is the size of the centralizer of $\be$, that is, if $\be$ is a permutation of cycle type  $(c_1, \dots, c_n)$, then $c(0, \be)=\prod_{i=1}^n i^{c_i}c_i!$. When $\be$ is an $m$-cycle we have that $c(0, \be)=m(n-m)!=m \times A000142(n-m)$, that we consider as a trivial relation with sequences in the OEIS database. In Table~\ref{centralizer} we present formulas for the centralizer of some $m$-cycles, where we have omitted the trivial cases.

Finally we consider two more cases. Let $\be_{m^{(2)}}$ denote a permutation of $2m$ symbols whose cycle factorization consists of a product of two $m$-cycles, and let $\be_{2^{(m)}}$ be a fixed-point free involution of  $2m$ symbols. In these cases we have
\begin{eqnarray*} 
c(0, \be_{m^{(2)}}) &=& 2m^2=A001105(m), m \geq 0.\\
c(0, \be_{2^{(m)}}) &=& 2^mm!=A000165(m), m \geq 0,
\end{eqnarray*}
 where A000165 is the double factorial of even numbers.
 
 It is possible that there are more sequences related to the formulas presented in this work. Also, we think that it is possible to obtain more formulas for $c(k, \be)$ for another special case of permutations. However, we believe that it is a difficult task to obtain a general formula for $c(k, \be)$. We have worked, without success, for the case of $m$-cycles, for every $m$, so we leave this as an open problem. Another future project is to find explicit bijections, as in the case of transpositions, between sets of $k$-commuting permutations and another combinatorial structures. 
\section{Acknowledgements} 
The author would like to thank D. Duarte, J. Lea\~nos and R. Moreno for some useful discussions and comments. The author would also like to thank the anonymous referee for his or her valuable comments and suggestions. This work was partially supported by PIFI-2014 and PROMEP (SEP, M\'exico) grants: UAZ-PTC-103 and UAZ-CA-169.

\end{document}